\documentclass[a4paper]{amsart}
\usepackage[active]{srcltx}
\usepackage[all]{xy}
\usepackage{hyperref}

\usepackage{amsmath}
\usepackage{amssymb,latexsym}
\usepackage{mathrsfs}
\usepackage{graphics}
\usepackage{latexsym}
\usepackage{import}
\usepackage{verbatim}
\usepackage[usenames]{color}
\usepackage{pifont,marvosym}

\theoremstyle{plain}
\newtheorem{lemma}{Lemma}[section]
\newtheorem{theorem}[lemma]{Theorem}

\theoremstyle{definition}

\numberwithin{equation}{section}

\newcommand{\supp}{\text{\rm supp}}

\newcommand{\ve}{\varepsilon}

\newcommand{\enne}{\mathbb{N}}

\newcommand{\gammA}{\boldsymbol{\gamma}}

\newcommand{\G}{\mathcal{G}}

\renewcommand{\r}{\varrho}

\begin{document}
\title[Self-intersection]{Self-Intersection of Optimal geodesics}

\author{Fabio Cavalletti and Martin Huesmann}

\address{RWTH, Department of Mathematics, Templergraben  64, D-52062 Aachen (Germany)}
\email{cavalletti@instmath.rwth-aachen.de}

\address{Universit\"at Bonn, Institut f\"ur angewandte Mathematik, Endenicher Allee 60, D-53115 Bonn (Germany)}
\email{huesmann@iam.uni-bonn.de}

\thanks{MH gratefully acknowledges funding by SFB 611}

\bibliographystyle{plain}

\begin{abstract}
Let $(X,d,m)$ be a geodesic metric measure space. Consider a geodesic $\mu_{t}$ in the $L^{2}$-Wasserstein space.
Then as $s$ goes to $t$  the support of $\mu_{s}$ and the support of $\mu_{t}$ have to overlap, provided an upper bound on the densities holds.
We give a more precise formulation of this self-intersection property. We consider for each $t$ the set of times for which a geodesic belongs to the support of 
$\mu_{t}$ and we prove that $t$ is a point of Lebesgue density 1 for this set, in the integral sense.
Our result applies to spaces satisfying $\mathsf{CD}(K,\infty)$. The non branching property is not needed.
\end{abstract}

\maketitle

\section{Introduction}

Let $(X,d,m)$ be a complete and separable metric measure space with the additional property that 
\begin{itemize}
\item $X$ coincide with the support of $m$; 
\item $(X,d)$ is a geodesic space.  
\end{itemize}
It is then well-known that the associated $L^{2}$-Wasserstein space $(\mathcal{P}_{2}(X),W_{2})$  is geodesic as well:
so to any $\mu_{0},\mu_{1} \in \mathcal{P}_{2}(X)$ we can associate a geodesic $[0,1] \ni t \mapsto \mu_{t}$ joining 
$\mu_{0}$ to $\mu_{1}$. See \cite{villa:Oldnew} for an overview (and much more) on optimal transportation.

Under some general assumption on the metric measure space $(X,d,m)$, like $\mathsf{CD}(K,\infty)$, see \cite{villott:curv}, \cite{sturm:MGH1}, \cite{sturm:MGH2} for their definition, 
it is possible to prove that if $\mu_{0}$ and $\mu_{1}$ are both absolute continuous with respect to $m$ with bounded densities, 
then the same property holds for the density of $\mu_{t}$, in particular $\mu_{t} \ll m$ and its density 
in bounded uniformly in $t\in [0,1]$, see for instance \cite{rajala:bounded}.

Thanks to this uniform bound on the density, by means of standard arguments in measure theory, one can prove that 
the support of $\mu_{t}$ has to overlap with $\mu_{0}$ as $t$ goes to $0$, otherwise to much ``mass'' would be present inside
the support of $\mu_{0}$. The same property holds for another time $s$ different from $0$ as $t$ goes to $0$. 
This overlapping property is to our knowledge the only 
qualitative property of the support of $\mu_{t}$ that has been proved so far. 
Even in the Euclidean framework, examples have been constructed showing that even assuming for instance convexity of 
$\supp[\mu_{0}]$ and $\supp[\mu_{1}]$, then $\supp[\mu_{t}]$ is not convex and in general one can expect 
$\supp[\mu_{t}]$ to be ``hardly'' disconnected.

Our interest is to give a more careful analysis of this overlapping property and 
to prove a structure property of $\supp[\mu_{t}]$. 
If we denote with $\G(X)\subset C([0,1];X)$ the subset of geodesics in $(X,d)$ and with 
$\mathcal{P}(\G(X))$ the space of probability measures over it,
it is again well-known that  to each geodesic $t \mapsto \mu_{t} \in \mathcal{P}_{2}(X)$ it is possible to associate 
$\gammA \in \mathcal{P}(\G(X))$, so that  
\[
(e_{t})_{\sharp} \gammA  = \mu_{t}, \qquad e_{t}: C([0,1];X) \to X, \quad e_{t}(\gamma) : = \gamma_{t}.
\]
for all $t \in [0,1]$, with $e_{t}$ the evaluation map at time $t$.
Our result will be stated in terms of the support of $\gammA$. 
So denote with $G \subset \G(X)$ the support of $\gammA$. 

For each $t \in [0,1]$ consider the set 
\[
I_{t}(\gamma) : = \{ \tau \in [0,1] : \gamma_{\tau} \in e_{t}(G)\},
\]
hence $I_{t}(\gamma)$ is the set of times for which $\gamma$ remains inside the support of $e_{t}(G)$.
So clearly $t \in I_{t}(\gamma)$. We will prove that if there exists a positive constant $C$ so that 
\[
\mu_{\tau} = \r_{\tau} m, \quad \r_{\tau} \leq C, 
\]
for all $\tau$ in a neighborhood of $t \in (0,1)$, then $t$ is 
a point of Lebesgue density 1 for $I_{t}(\gamma)$ in the $L^{1}(\gammA)$-sense
that is
\[
\lim_{\ve \to 0}  \frac{\mathcal{L}^{1}\big(I_{t}(\gamma) \cap (t- \ve, t+ \ve ) \big)  }{2\ve} =1, \qquad \textrm{in  } L^{1}(G,\gammA). 
\]

\section{The Result}

So let $(X,d,m)$ be a metric measure space verifying the assumption stated before. 
Let $\mu_{0},\mu_{1} \in \mathcal{P}_{2}(X)$ and $t\mapsto \mu_{t}\in \mathcal{P}_{2}(X)$ a geodesic connecting them.
Moreover $\gammA \in \mathcal{P}(\G(X))$ is the dynamical optimal plan associated to $\mu_{t}$.

\begin{theorem}
Fix $t \in (0,1)$ and assume the existence of a positive constant $C$ and a neighborhood $U_{t} \subset [0,1]$ of $t$ such that 
$\mu_{\tau} = \r_{\tau}m$ with $\r_{\tau} \leq C$ for each $\tau \in U_{t}$.
Then 
\begin{equation}\label{E:density}
\lim_{\ve \to 0}  \frac{\mathcal{L}^{1}\big(I_{t}(\gamma) \cap (t- \ve, t+ \ve ) \big)  }{2\ve} =1, 
\end{equation}
in $L^{1}(G,\gammA)$. 
\end{theorem}

\begin{proof}
Assume by contradiction  \eqref{E:density} doesn't hold. 
Therefore there exists a set $H \subset G$, with $\gammA(H)>0$ so that for all $\gamma \in H$ we have two possibilities: or
\[ 
0 \leq \liminf_{\ve \to 0}  \frac{\mathcal{L}^{1}\big(I_{t}(\gamma) \cap (t- \ve, t+ \ve ) \big)  }{2\ve}
< \limsup_{\ve \to 0}  \frac{\mathcal{L}^{1}\big(I_{t}(\gamma) \cap (t- \ve, t+ \ve ) \big)  }{2\ve}  \leq 1. 
\]
either the limit exists but is not one: 
\[
\lim_{\ve \to 0}  \frac{\mathcal{L}^{1}\big(I_{t}(\gamma) \cap (t- \ve, t+ \ve ) \big)  }{2\ve} <1.
\]
In the first part of the proof we show that in the first case, the $\limsup$ must be equal to $1$, neglecting a set of zero $\gammA$-measure. 
The same argument excludes immediately the second case.

{\it Step 1.} Suppose by contradiction the existence of a set $H \subset G$, with $\gammA(H) >0$ so that for all $\gamma \in H$
\[
\limsup_{\ve \to 0}  \frac{\mathcal{L}^{1}\big(I_{t}(\gamma) \cap (t- \ve, t+ \ve ) \big)  }{2\ve} <1.
\]
and therefore possibly restricting $H$,
\[
\liminf_{\ve \to 0}  \frac{\mathcal{L}^{1}\big(I^{c}_{t}(\gamma) \cap (t- \ve, t+ \ve ) \big)  }{2\ve} > \alpha,
\]
for some $\alpha >0$. Let
\[
E: = \{ (\gamma,s) \in H \times (0,1) : t+s \in I_{t}(\gamma)^{c} \}=\{ (\gamma,s) \in H \times (0,1) : d(\gamma_{t+s},e_{t}(G)) > 0 \}.
\]
Then by Fubini's Theorem
\[
\gammA \otimes \mathcal{L}^{1}(E) = \int_{(0,1)}  \mathcal{L}^{1}(E(\gamma))\gammA(d\gamma), \qquad E(\gamma):= P_{2}\Big( E \cap \big( \{\gamma\} \times (0,1) \big) \Big),
\]
where $P_{i}$ denotes the projection map on the $i$-th component, for $i = 1,2$. 
From Fato\'u's Lemma 
\[
\liminf_{\ve \to 0} \frac{\gammA \otimes \mathcal{L}^{1} \Big( E \cap \big(H \times (t - \ve, t+\ve)\big) \Big)}{2\ve} \geq \alpha \gammA(H),
\]
therefore
\[
\liminf_{\ve \to 0} \frac{1}{2\ve} \int_{(t-\ve,t+\ve)} \gammA(E(\tau)) \mathcal{L}^{1}(ds) \geq \alpha \gammA(H), \qquad 
E(\tau) : = P_{1}\Big( E \cap \big( H \times \{\tau\} \big) \Big).
\]
So there must be a sequence of $\{s_{n}\}_{n \in \enne}$ converging to $0$ so that $\gammA(E (t+ s_{n})) \geq C$, for some $C>0$.
Then, since $e_{t+s_{n}}(G)$ converges to $e_{t}(G)$ in Hausdorff topology as $s_{n}$ goes to 0, we have
\[
m( e_{t}(G)^{\ve}) \geq m( e_{t}(G) \cup  e_{t+ s_{n}}(E(t+s_{n}))) \geq m(e_{t}(G)) + m(e_{t+s_{n}}(E(t+s_{n}))),
\]
where $e_{t}(G)^{\ve} := \{z \in X: d(z,e_{t}(G))\leq\ve \}$.
Since by assumption $\r_{\tau} \leq C$ on $e_{\tau}(G)$ for all $\tau \in U_{t}$, it follows that $m(e_{t+s_{n}}(E(t+s_{n})))$ remains uniformly strictly positive as $s_{n}$ goes to $0$. 
Since 
\[
m(e_{t}(G)) \geq \limsup_{\ve \to 0}m(e_{t}(G)^{\ve}),
\]
we have a contradiction.  Hence we have shown that there exists $H$, $\gammA$-negligible, so that
\[
\limsup_{\ve \to 0}  \frac{\mathcal{L}^{1}\big(I_{t}(\gamma) \cap (t- \ve, t+ \ve ) \big)  }{2\ve} =1,
\]
for all $\gamma \in G\setminus H$. Using the same reasoning we can also prove a stronger statement: for any 
sequence $\ve_{n} \to 0$ there exists $H$, $\gammA$-negligible and depending on the sequence $\ve_{n}$, so that 
\[
\limsup_{n\to \infty}  \frac{\mathcal{L}^{1}\big(I_{t}(\gamma) \cap (t- \ve_{n}, t+ \ve_{n} ) \big)  }{2\ve_{n}} =1,
\]
for all $\gamma \in G\setminus H$.
We now show that $L^{1}$-convergence holds.

{\it Step 2.} Consider any sequence $\ve_{n}$ converging to $0$. Then from the equality 
\[
1 - \frac{\mathcal{L}^{1} (I_{t}^{c}(\gamma) \cap (t-\ve_{n}, t+\ve_{n})) }{2\ve_{n}} = 
\frac{\mathcal{L}^{1} (I_{t}(\gamma) \cap (t-\ve_{n}, t+\ve_{n})) }{2\ve_{n}},
\]
integrating over any set $K \subset G$ we get
\begin{align*}
\gammA(K) 
- &~ \limsup_{n\to \infty} \int_{K} \frac{\mathcal{L}^{1} (I_{t}^{c}(\gamma) \cap (t-\ve_{n}, t+\ve_{n})) }{2\ve_{n}} \gammA(d\gamma)\crcr
= &~ \liminf_{n\to \infty} \int_{K} \frac{\mathcal{L}^{1} (I_{t}(\gamma) \cap (t-\ve_{n}, t+\ve_{n})) }{2\ve_{n}} \gammA(d\gamma).
\end{align*}
Hence, if there exists $K\subset G$ with $\gammA(K)>0$ so that 
\[
\liminf_{n\to \infty} \int_{K} \frac{\mathcal{L}^{1} (I_{t}(\gamma) \cap (t-\ve_{n}, t+\ve_{n})) }{2\ve_{n}} \gammA(d\gamma) < \gammA(K),
\]
then
\[
\limsup_{n\to \infty} \int_{K} \frac{\mathcal{L}^{1} (I_{t}^{c}(\gamma) \cap (t-\ve_{n}, t+\ve_{n})) }{2\ve_{n}} \gammA(d\gamma)>0.
\]
Using Fubini's Theorem as in {\it Step 1.} we could find a sequence of slices of $E$, say $E(t+s_{n})$, with uniformly positive 
$\gammA$-measure and therefore we would get a contradiction. Reasoning in the same manner for the $\limsup$ we get that 
for any $K \subset G$, 
\[
\liminf_{n\to \infty} \int_{K} \frac{\mathcal{L}^{1} (I_{t}(\gamma) \cap (t-\ve_{n}, t+\ve_{n})) }{2\ve_{n}} \gammA(d\gamma) = \gammA(K).
\]
Since $\ve_{n}$ was arbitrarily chosen
\[
\lim_{\ve \to 0 }\frac{\mathcal{L}^{1} (I_{t}(\gamma) \cap (t-\ve, t+\ve)) }{2\ve} = 1, \qquad \textrm{in} \ L^{1}(\gammA\llcorner_{G}).
\]
and the claim follows. 
\end{proof}


\begin{thebibliography}{1}

\bibitem{villott:curv}
J.~Lott and C.~Villani.
\newblock Ricci curvature for metric-measure spaces via optimal transport.
\newblock {\em Ann. of Math.}, 169(3):903--991, 2009.

\bibitem{rajala:bounded}
T.~Rajala.
\newblock Interpolated measures with bounded density in metric spaces
  satisfying the curvature-dimension conditions of {S}turm.
\newblock {\em J. Funct. Anal.}, 263:896--924, 2012.

\bibitem{sturm:MGH1}
K.T. Sturm.
\newblock On the geometry of metric measure spaces.{I}.
\newblock {\em Acta Math.}, 196(1):65--131, 2006.

\bibitem{sturm:MGH2}
K.T. Sturm.
\newblock On the geometry of metric measure spaces.{II}.
\newblock {\em Acta Math.}, 196(1):133--177, 2006.

\bibitem{villa:Oldnew}
C.~Villani.
\newblock {\em Optimal transport, old and new}.
\newblock Springer, 2008.

\end{thebibliography}
\end{document}